\documentclass[a4,10pt]{scrartcl}
\usepackage{typearea}
\typearea{12}
\usepackage{amsmath,amsfonts,amsthm,amssymb}
\usepackage{color}
\usepackage{algorithmic,algorithm}
\usepackage{framed}
\usepackage{arydshln}
\usepackage{url}
\usepackage[normalem]{ulem}
\usepackage[colorlinks,linktocpage,linkcolor=blue]{hyperref}

\newtheorem{theorem}{Theorem}[section]
\newtheorem{proposition}{Proposition}[section]
\newtheorem{lemma}{Lemma}[section]
\newtheorem{corollary}{Corollary}[section]

\numberwithin{equation}{section}
\theoremstyle{definition}
\newtheorem{definition}[theorem]{Definition}
\newtheorem{remark}[theorem]{Remark}
\newtheorem{example}[theorem]{Example}

\newcommand{\R}{\mathbb{R}}
\newcommand{\prox}[1]{\mbox{\rm prox}_{#1}}
\newcommand{\phic}{\frac{\phi}{c}}
\newcommand{\mi}{\mathbf{i}}
\newcommand{\mo}{\mathbf{o}}

\title{Lagrange optimality system for a class of nonsmooth convex optimization.}
\author{
Bangti Jin\thanks{Department of Computer Science, University College London, Gower Street, London WC1E 6BT, UK. (bangti.jin@gmail.com)}
\and 
Tomoya Takeuchi\thanks{Institute of Industrial Science, The University of Tokyo, Tokyo, Japan. (takeuchi@sat.t.u-tokyo.ac.jp)}}
\date{}
\begin{document}
\maketitle

\begin{abstract}
In this paper, we revisit the augmented Lagrangian method for a class of nonsmooth convex optimization. 
We present the Lagrange optimality system of the augmented Lagrangian associated with the problems,
and establish its connections with the standard optimality condition and the saddle point condition of the augmented Lagrangian,
which provides a powerful tool for developing numerical algorithms.
We apply a linear Newton method to the Lagrange optimality system to obtain a novel algorithm applicable to a variety of nonsmooth convex optimization problems arising in practical applications.
Under suitable conditions, we prove the nonsingularity of the Newton system and the local convergence of the algorithm.
\end{abstract}
{\footnotesize}
\section{Introduction}
In this paper we consider the augmented Lagrangian method for solving a class of nonsmooth convex optimization problems
\begin{equation}\label{min}
 \min_{x\in X} f(x) + \phi(E x),
\end{equation}
where the function $f \colon X \rightarrow \R$ is convex and continuously differentiable on a Banach space $X$, $\phi \colon
H \rightarrow \R^+$ is a proper, lower semi-continuous and convex function on a Hilbert space $H$, and $E$ is a
bounded linear operator from $X$ to $H$. We assume that the proximity operator of the convex function $\phi$ has a closed form expression.
This problem class encompasses a wide range of optimization problems arising in practical applications, e.g., inverse problems, variational problems, image processing, signal processing
and statistics to name a few
\cite{boyd2011distributed,ChanShen:2005,Combettes+Wajs-Signrecoproxforw:05,Ekeland+Temam-Convanalvariprob:99,Glowinski-Numemethnonlvari:08,ItoJin:2014,Ito+Kunisch-Lagrmultapprvari:08}.

The augmented Lagrangian method was proposed independently by Hestenes \cite{Hestenes-Multgradmeth:69} and Powell
\cite{Powell-methnonlconsmini:69} for solving nonlinear programming problems
with equality constraints.
The method was studied in relation to Fenchel duality and generalized to nonlinear programming problems with inequality constraints by Rockafellar \cite{Rockafellar-dualapprsolvnonl:73,Rockafellar-multmethHestPowe:73}. Later it was further generalized to the problem~\eqref{min} by Glowinski and Marroco \cite{glowinski1975} where the augmented Lagrangian is given by
\begin{equation*}
\mathcal{L}_c(x,v,\lambda) = f(x) + \phi(v) + (\lambda, E x - v) + \frac{c}{2}\Vert E x - v\Vert^2.
\end{equation*}
The inner product $(\lambda,E x-v)$ dualizes the equality constraint, and the quadratic term penalizes the
constraint violation for the following equality constrained problem equivalent to problem~\eqref{min}:
\begin{equation*}
  \min_{x\in X,v\in H} f(x) + \phi(v) \quad \mbox{ subject to } \quad E x = v.
\end{equation*}
A solution of problem~\eqref{min} can be characterized, under certain conditions on $f$, $\phi$ and $E$, as a saddle point of the augmented Lagrangian, and
the strong duality theorem leads to first-order algorithms for the dual function $\theta(\lambda)=\inf_{x,v}\mathcal{L}_c(x,v,\lambda)$.
In practical implementation, the combination of the dualization and the penalization alleviates the slow convergence for the ordinary Lagrangian methods and ill conditioning as $c\rightarrow \infty$ for penalty methods. Due to these advantages over the standard Lagrangian formulation and the penalty formulation, a large number of first order algorithms based on the augmented Lagrangian $\mathcal{L}_c$ have been developed for a wide variety of applications; see e.g., \cite{boyd2011distributed,GlowinskiLeTallec:1989,parikh2013proximal,WuTai:2010}.

An  alternative Lagrangian for~\eqref{min} has been introduced by Fortin~\cite{Fortin-Minisomenon-func:75}, which was obtained by employing the partial conjugate of the augmented perturbation bifunction $F_c(x,v) = f(x) + \phi(E x - v) + \frac{c}{2}\Vert v\Vert^2$ due to Rockafeller \cite{Rockafellar-dualapprsolvnonl:73}:
\begin{align*}
  L_c(x,\lambda) &= \min_{v\in H}\left( (v,\lambda) + F_c(x,v) \right)
  =\min_{v\in H}\left( (v,\lambda) + f(x) + \phi(E x - v) + \frac{c}{2}\Vert v\Vert^2\right)  \nonumber \\
  &= \min_{u\in H}\left( (E x - u,\lambda) + f(x) + \phi(u) + \frac{c}{2}\Vert E x - u\Vert^2\right)  \nonumber  \\
  & =f(x) + \min_{u\in H}\left( \phi(u) + (\lambda,E  x - u) + \frac{c}{2}\Vert E x -u\Vert^2\right) \nonumber  \\
  & = f(x) + \phi_c(E x + \lambda/c) -\tfrac{1}{2c}\|\lambda\|^2,  
\end{align*}
where $c$ is a positive constant and the function $\phi_c(z)$ is the Moreau envelope (see Section~\ref{sec:proximity} for the definition). It was shown that a saddle point of $L_c$ is also a saddle point of the standard Lagrangian and conversely \cite[Thm.~2.1]{Fortin-Minisomenon-func:75}.
A first order algorithm often referred to as the \textit{augmented Lagrangian algorithm}, which is quite similar to the one developed in \cite{glowinski1975}, was proposed for certain special cases of the function $\phi$ \cite[Thm.~4.1]{Fortin-Minisomenon-func:75}.
The augmented Lagrangian method was further studied by Ito and Kunisch \cite{Ito+Kunisch-AugmLagrmethnons:00}
for the following optimization problem
\begin{equation}\label{minc}
 \min_{x\in C} f(x) + \phi(E x),
\end{equation}
where $C$ is a convex set in $X$. One of their major achievements is the results concerning the existence of a Lagrange multiplier for problem \eqref{minc}:
It was shown that under appropriate conditions Lagrange multipliers of a regularized problem defined by the augmented Lagrangian $L_c$ converge and the limit is a Lagrange multiplier of problem \eqref{minc}.
In addition to the valuable contribution, the augmented Lagrangian algorithm by Fortin was extended to a more general class of convex functions $\phi$,
and the convergence of the algorithm was established. It is noted that the problem can be reformulated into problem~\eqref{min}, by redefining the convex function $\phi$ and the linear map $E$ by
$\phi(x,y) := \phi(x) + \chi_C(y)$ and $E x := (E x, x)$, respectively, where $\chi_C$ is the
characteristic function of the convex set $C$. Hence, it shares an identical structure with problem~\eqref{min}.

The augmented Lagrangian $L_c$ is Fr\'{e}chet differentiable, cf. Section~\ref{sec:optimality}, which
motivates the use of the Lagrange optimality system
\begin{equation}\label{sys:lag}
D_{x}L_c(x,\lambda) = 0,\quad \mbox{and}\quad \quad D_{\lambda}L_c(x,\lambda) = 0,
\end{equation}
to characterize the saddle point and hence the solution of problem~\eqref{min}.
This perspective naturally leads to the application of Newton methods for solving the nonlinear system. However, the Moreau envelope involved
in~\eqref{sys:lag}, cf. Proposition~\ref{prop:KKT}, is twice continuously differentiable if and only if the same is true
for the convex function $\phi$ \cite{Lemarechal+Sagastizabal-PracaspeMoreregu:97}, and thus the standard (classical)
Newton methods cannot be applied directly to the Lagrange optimality system.
Semismooth Newton methods and quasi-Newton
methods are possible alternatives for solving the Lagrange optimality system, but there are some drawbacks in their
applications to the Lagrange optimality system: The inclusion appearing in the chain rule of a composite
map makes it difficult to theoretically identify a generalized or limiting Jacobian of $D_{x,\lambda}L_c$ for semismooth
Newton methods, while the superlinear convergence of quasi-Newton methods holds only when the system to be
solved is differentiable at the solution \cite{Ip+Kyparisis-Locaconvquasmeth:92}.
We opt for instead linear Newton methods \cite{Facchinei+Pang-Finivariineqcomp:03a} to solve the Lagrange optimality system~\eqref{sys:lag},
where one replaces the generalized Jacobian of $D_{x,\lambda}L_c$ in semismooth Newton methods with a \textit{linear Newton approximation (LNA)} of $D_{x,\lambda}L_c$.
Calculus rules, which provide a systematic way of generating LNAs of a given map,
reduce the construction of a LNA of the Lagrange optimality system to the computation of the (Clarke's) generalized or limiting Jacobian of the proximity operator involved in the system, cf. Section~\ref{sec:newton}.

The focus of this work is twofold.
First, we present the Lagrange optimality system,
which was not provided in both \cite{Fortin-Minisomenon-func:75} and \cite{Ito+Kunisch-AugmLagrmethnons:00},
and establish its connection with the standard optimality system of problem~\eqref{min} and the saddle point condition of the augmented Lagrangian.
Second, we develop a Newton type algorithm for the Lagrange optimality system.
To the best of our knowledge, this is the first work using the Lagrange optimality system for developing
Newton type algorithms for nonsmooth convex optimization~\eqref{min}.
These two aspects represent the essential contributions of this work.

The rest of the paper is organized as follows. In Section~\ref{sec:proximity} we collect fundamental results on the
Moreau envelope and the promixity operator, which provide the main tools for the analysis. In Section~\ref{sec:optimality},
we investigate the connection among the optimality system for the problem~\eqref{min}, the Lagrange optimality system and the saddle point of the augmented Lagrangian $L_c$.
In Section~\ref{sec:newton}, we develop a Newton method for problem~\eqref{min}, which exhibits a local Q-superlinear convergence. 

\subsection{Notations}
We denote by $X$ a real Banach space with the norm $| \cdot|$.
The duality bracket between the dual space $X^\ast$ and $X$ is denoted by $\langle\cdot,\cdot\rangle_{X^\ast,X}$.
For a twice continuously differentiable function $f$, its derivative is denoted by $Df(x)$ or $D_xf(x)$, and its Hessian by $D^2_xf(x)$. $H$ is a Hilbert space
with the inner product $(\cdot,\cdot)$, and the norm on $H$ is denoted by $\| \cdot\|$. The set of proper, lower
semicontinuous, convex functions defined on the Hilbert space $H$ is denoted by $\Gamma_0(H)$. The effective domain
of a function $\phi\in\Gamma_0(H)$ is denoted by $\mathcal{D}(\phi)=\{z\in H\mid \phi(z)\mbox{ is finite}\}$, and
it is always assumed to be nonempty.
For a function $\phi\in \Gamma_0(H)$, the convex conjugate $\phi^*$ is defined by $\phi^{\ast}(z^\ast) = \sup_{z\in H}\left( (z^{\ast}, z) - \phi(z) \right)$.
A \textit{subgradient} of $\phi$ at $x\in H$ is $g\in H$ satisfying
\begin{equation*}
\phi(y)\ge \phi(x) + (g, y-x),\quad \forall y \in H.
\end{equation*}
The \textit{subdifferentials} of $\phi$ at $x$ is the set of all subgradients of $\phi$ at $x$, and is denoted by $\partial \phi(x)$.
\section{Moreau envelope and proximity operator}\label{sec:proximity}
The central tools for analyzing the augmented Lagrangian approach are Moreau envelope and proximity operator.
We recall their definitions and basic properties that are relevant to the development of the Lagrange multiplier theory.
We note that for $\phi\in\Gamma_0(H)$ the strictly convex function $u\rightarrow \phi(u) + \frac{1}{2}\Vert u-z\Vert^2$
admits a unique minimizer.

\begin{definition}
Let $\phi\in \Gamma_0(H)$ and $c>0$.
The \textit{Moreau envelope} $\phi_c :H \rightarrow \R$ and the \textit{proximity operator} $ \prox{\phi}: H\rightarrow H$  are defined respectively as
\begin{align*}
  \phi_c(z)& = \min_{u\in H}\left(\phi(u)+\tfrac{c}{2}\Vert u- z\Vert^2 \right),\\
\prox{\phi}(z)&=\arg\min_{u \in H} \left(\phi(u) + \tfrac{1}{2}\Vert u-z\Vert^2 \right),
\end{align*}
for $z\in H$.
\end{definition}
By definition we have
\begin{equation*}
  \begin{aligned}
  \prox{\frac{\phi}{c}}(z) = \arg \min_{u \in H}\left(\tfrac{\phi(x)}{c}+\tfrac{1}{2}\Vert u-z\Vert^2\right)=\arg \min_{u\in H} \left(\phi(u)+\tfrac{c}{2}\Vert u-z\Vert^2  \right),
  \end{aligned}
\end{equation*}
and
\begin{equation*}
  \phi_c(z) =\phi(\prox{\phic}(z))+\tfrac{c}{2}\Vert \prox{\phic}(z)- z\Vert^2.
\end{equation*}
We refer interested readers to Tables~10.1
and 10.2  of \cite{Combettes+Wajs-Signrecoproxforw:05} for closed-form expressions of a number of frequently used proximity operators.

We recall well-known properties of the Moreau envelope and proximity operator.
\begin{proposition}[\protect{\cite{Bauschke+Combettes-Convanalmonooper:11}}]\label{prop:basic}
Let $z\in H$ and $c>0$. Let $\phi\in \Gamma_0(H)$.
\begin{enumerate}
\item[{\rm (a)}]  $0 \le \phi(z)-\phi_c(z) $ for all $z \in H$ and all $c>0$.
\item[{\rm (b)}]  $\lim_{c\to \infty}\phi_c(z) = \phi(z)$ for all $z\in H$.
\item[{\rm (c)}] The proximity operator $\prox{\phic}$ is nonexpansive, that is,
\begin{equation*} 
 \Vert \prox{\phic}(z) - \prox{\phic}(w)\Vert^2 \le (\prox{\phic}(z) - \prox{\phic}(w), z-w), \quad \forall z, \forall w\in H.
 \end{equation*}
 \item[{\rm (d)}] The Moreau envelope $\phi_c$ is Fr\'echet differentiable and the gradient is given by
\begin{equation}\label{phip}
   D_z\phi_c(z) =c(z - \prox{\phic}(z) ), \quad \forall c>0, \forall z\in H.
\end{equation}
\item[{\rm (e)}] The gradient $z \rightarrow D_z\phi_c(z) \in H$ is Lipschitz continuous with a Lipschitz constant $c$, i.e.,
\begin{equation*} 
\Vert D_z \phi_c(z) -  D_z \phi_c(w)\Vert \le c\Vert z -w\Vert,  \quad \forall z ,\forall w \in H.
\end{equation*}
\item[{\rm (f)}] The Moreau envelope and the proximity operator of the conjugate of $\phi$ are related with $\phi_c$ and $\prox{\phic}$, respectively as
\begin{equation*} 
\phi_c(z) + (\phi^\ast)_{\frac{1}{c}}\left(cz\right) =\tfrac{c}{2}\Vert z \Vert^2, \quad
\prox{\phic}(z) + \tfrac{1}{c}\prox{c\phi^{\ast}}\left(cz\right) = z.
\end{equation*}
\end{enumerate}
\end{proposition}
All the results are standard; The proofs can be found in e.g., \cite{Bauschke+Combettes-Convanalmonooper:11}.
Here we give an alternative proof of {\rm (f)} based on the duality theory.
\begin{proof}
For $z\in H$, we define the function $L_z\colon H\times \mathcal{D}(\phi)\rightarrow\R$ by
\begin{equation*}
L_z(u,p):=(u,p)-\phi(p)+\tfrac{1}{2c}\Vert u - cz\Vert^2.
\end{equation*}
Clearly, $L_z$ is convex in $u$ and is concave in $p$. We claim that $L_z$ posses a saddle point on
$H\times \mathcal{D}(\phi)$. Clearly, $\lim_{\Vert u\Vert \to \infty}L_z(u,p)=\infty$ for all
$p\in \mathcal{D}(\phi)$. Thus by \cite[Chap.~6, Prop.~2.3]{Ekeland+Temam-Convanalvariprob:99}, we have
\begin{equation}\label{minmax0}
\inf_u\sup_p L_z(u,p) = \sup_p\inf_u L_z(u,p).
\end{equation}
Now we compute $\inf_u\sup_p L_z(u,p)$ and $\sup_p \inf_u L_z(u,p)$ separately. First, we observe
\begin{align*}
\inf_u\sup_p L_z(u,p)&=\inf_{u}\left(\sup_{p}((u,p)-\phi(p)) + \tfrac{1}{2c}\Vert u-cz\Vert^2 \right)\\
  &=\inf_{u}\left(\phi^\ast(u) + \tfrac{1}{2c}\Vert u-cz\Vert^2\right)  =  (\phi^\ast)_{\frac{1}{c}}(cz).
\end{align*}
Meanwhile, we have
\begin{align*}
  \inf_{u}L_z(u,p)& = \inf_{u}\left(\tfrac{1}{2c}\Vert u -cz\Vert^2 + (p,u)\right)-\phi(p)
  =\tfrac{1}{2c}\Vert cp\Vert^2 + (p,c(z-p)) -\phi(p)\\
  &=c(p,z)-\tfrac{c}{2}\Vert p\Vert^2-\phi(p)=\tfrac{c}{2}\Vert z\Vert^2 -\left(\phi(p)+\tfrac{c}{2}\Vert p - z\Vert^2\right).
\end{align*}
Thus, we deduce
\begin{equation*}
\sup_p\inf_u L_z(u,p) = \sup_{p}\left(\tfrac{c}{2}\Vert z\Vert^2 -\left(\phi(p)+\tfrac{c}{2}\Vert p - z\Vert^2\right)\right)=\tfrac{c}{2}\Vert z\Vert^2 -\phi_c(z).
\end{equation*}
Therefore, from~\eqref{minmax0} we have
\begin{equation*}
  (\phi^\ast)_{\frac{1}{c}}(cz)  =\inf_{u}\sup_{p}L_z(u,p)=\sup_{p}\inf_{u}L_z(u,p)   =\tfrac{c}{2}\Vert z\Vert^2 -\phi_c(z),
\end{equation*}
which shows the first relation. Differentiating both side of this equation with respect to $z$ and using \eqref{phip} result in the second relation.
\end{proof}
The Moreau envelope and the proximity operator provide equivalent expressions of the inclusion
$\lambda\in\partial\phi(z)$.
\begin{proposition}\label{prop:proximity2}
Let $c>0$ be an arbitrary fixed constant and $\phi\in\Gamma_0(H)$.
Then the following conditions are equivalent.
\begin{enumerate}
\item[{\rm (a)}] 
$\lambda \in \partial \phi(z)$.
\item[{\rm (b)}] 
$z - \prox{\phic}(z+\lambda/c)=0$.
\item[{\rm (c)}] 
$\phi(z) = \phi_c(z+\lambda/c)-\frac{1}{2c}\Vert \lambda \Vert^2$.
\end{enumerate}
\end{proposition}
\begin{proof}
Let the pair $(z,\lambda)$ satisfy the condition $\lambda \in \partial \phi(z)$. This can be expressed as
\begin{equation*}
0 \in \partial \phi(z) + c(z - (z+\lambda/c)) = \partial_u \left(\phi(u) + \frac{c}{2}\Vert u - (z+\lambda/c)\Vert^2\right)|_{u=z},
\end{equation*}
which is equivalent to $z = \prox{\phic}
(z+\lambda/c)$. This shows the equivalence between {\rm (a)} and {\rm (b)}.
Next we show that {\rm (b)} implies {\rm (c)}. Suppose $z-\prox{\phic}(z+\lambda/c)=0$.
Then by the definition of $\phi_c$, it follows that
\begin{align*}
\phi_c(z+\lambda/c) &= \phi(\prox{\phic}(z+\lambda/c)) + \tfrac{c}{2}\Vert \prox{\phic}(z+\lambda/c)-(z+\lambda/c)\Vert^2\\
& = \phi(z) + \tfrac{c}{2}\Vert z-(z+\lambda/c)\Vert^2 = \phi(z) + \tfrac{1}{2c}\Vert \lambda\Vert^2.
\end{align*}
Finally, we show that {\rm (c)} implies {\rm (a)}. By the definition of the Moreau envelope, it follows that
\begin{equation*}
  \phi_c(z+\lambda/c)\le  \phi(u) + \frac{c}{2}\Vert u-(z+\lambda/c)\Vert^2, \qquad \forall u\in H,
\end{equation*}
which is equivalently written as
\begin{equation*}
\phi(z) = \phi_c(z+\lambda/c)-\frac{1}{2c}\Vert \lambda\Vert^2\le \phi(u) + \frac{c}{2}\Vert u-z\Vert^2 + (u-z,-\lambda),\qquad \forall u\in H.
\end{equation*}
This implies that the strictly convex function $u\rightarrow \phi(u) + \frac{c}{2}\Vert u-z\Vert^2 + (u-z,-\lambda)$ attains its minimum at $z$.
Thus
\begin{align*}
0\in \partial_u\left(\phi(u) + \frac{c}{2}\Vert u-z\Vert^2 + (u-z,-\lambda)\right)|_{u=z} =
 \partial \phi(u)- \lambda,
\end{align*}
which proves that {\rm (c)} implies {\rm (a)}.
\end{proof}
\section{The optimality systems}\label{sec:optimality}
In the classical optimization problem for a smooth cost function with equality constraints by smooth maps, it is well known that saddle points are characterized by Lagrange optimality system of the (standard) Lagrangian associated with the optimization problem. In this section, we show that the augmented Lagrangian $L_c$ generalizes the classical result to the nonsmooth convex optimization problem~\eqref{min}.
\begin{proposition}\label{prop:KKT}
Let $c>0$, $f$ be convex and continuously differentiable, and $\phi\in\Gamma_0(H)$. The augmented Lagrangian
$L_c$ satisfies the following properties.
\begin{itemize}
 \item[{\rm (a)}] $L_c$ is finite for all $x\in X$ and for all $\lambda \in H$.
 \item[{\rm (b)}] $L_c$ is convex and continuously differentiable with respect to $x$, and is concave and continuously
 differentiable with respect to $\lambda$. Further, for all $(x,\lambda)\in X\times H$ and for all $c>0$, the gradients $D_xL_c$ and $D_{\lambda} L_c$ are written respectively as
\begin{align}
  & D_{x} {L}_{c}(x,\lambda) = D_x f(x) +c E^{\rm T}
   (E x + \lambda/c -\prox{\phic}(E x + \lambda/c)), \label{KKT1}\\
  & D_{\lambda}{L}_{c}(x,\lambda) = E x - \prox{\phic}(E x  + \lambda/c). \label{KKT2}
\end{align}
\item[{\rm (c)}] $D_{x}L_c(x,\lambda)$ can be expressed in terms of $D_{\lambda}L_c(x,\lambda)$ by
\begin{equation}\label{Lx}
  D_xL_{c}(x,\lambda) = D_xf(x)+E^{\rm T}\left(\lambda + cD_{\lambda}L_{c}(x,\lambda)\right).
\end{equation}
\end{itemize}
\end{proposition}
\begin{proof}
All the assertions follow directly from the differentiability and convexity of $f$, and Proposition~\ref{prop:basic}.
\end{proof}
\begin{theorem}\label{thm:sol_multi}
Let ${c}>0$, $f$ be convex and continuously differentiable, and $\phi\in\Gamma_0(H)$. The following
conditions on a pair $(\bar{x},\bar{\lambda})$ are equivalent.
\begin{itemize}
 \item[{\rm (a)}] (optimality system) A pair $(\bar{x},\bar \lambda) \in X\times H$ satisfies the optimality system
\begin{align}\label{eqn:extremal}
D_xf(\bar{x}) + E^{{\rm T}} \bar{\lambda} = 0 \quad \mbox{and}\quad \bar{\lambda}\in \partial\phi(E\bar{x}).
\end{align}
\item[{\rm (b)}] (Lagrange optimality system) A pair $(\bar{x},\bar\lambda) \in X\times H$ satisfies the Lagrange optimality system
\begin{equation}\label{KKT}
  D_x L_{c}(\bar{x},\bar{\lambda})=0\quad\mbox{and}\quad
  D_{\lambda} L_{c}(\bar{x},\bar{\lambda})=0,
\end{equation}
where the gradients of $L_c$ with respect to $x$ and $\lambda$ are given by~\eqref{KKT1} and~\eqref{KKT2}, respectively.
More precisely, $(\bar{x},\bar{\lambda})$ satisfies the nonlinear system:
\begin{equation*}
\left\{\begin{array}{l}
  D_x f(x) + cE^{\rm T}
\left(E x + \lambda/c -\prox{\phic}(E x + \lambda/c)\right)=0\\[5pt]
 E x - \prox{\phic}(E x  +\lambda/c)=0.
\end{array} \right.
\end{equation*}
\item[{\rm (c)}](saddle point) A pair $(\bar{x},\bar\lambda)\in X\times H$ is a saddle point of $L_{c}$:
\begin{equation}\label{eqn:saddle}
L_c(\bar{x},{\lambda}) \le L_c(\bar{x},\bar{\lambda}) \le
L_c(x,\bar{\lambda}),\qquad \forall x\in X,\; \forall\lambda \in H.
\end{equation}
\end{itemize}
\end{theorem}
\begin{proof}
First we show the equivalence between {\rm (a)} and {\rm (b)}. Suppose that {\rm (a)} holds.
The inclusion $\lambda\in \partial\phi(E x)$ is equivalent to the equation $E x-\prox{\phic}
(E x + \lambda/c)=0$ by Propostion~\ref{prop:proximity2}. Hence, from~\eqref{KKT2} we have
\begin{equation*}
D_{\lambda}L_{{c}}(\bar{x},\bar{\lambda})=E\bar{x} - \prox{\phic}(E \bar{x}+\bar{\lambda}/c) = 0.
\end{equation*}
Thus
\begin{align*}
D_{x}L_{c}(\bar{x},\bar{\lambda})
= D_x f(\bar{x}) + E^{{\rm T}}\left(\bar{\lambda}+c D_{\lambda}L_{c}(\bar{x},\bar{\lambda})\right)
= D_x f(\bar{x}) + E^{\rm T} \bar{\lambda} = 0,
\end{align*}
by Proposition~\ref{prop:KKT}{\rm (c)}.
Similarly, we can show that {\rm (b)} implies {\rm (a)}.

Next we show the equivalence between {\rm (b)} and {\rm (c)}. If $(\bar{x},\bar{\lambda})$ satisfies the Lagrange optimality system, then from the
convexity of $L_{c}(x,\lambda)$ with respect to $x$, we have
\begin{equation*}
L_{c}(x,\bar{\lambda})-L_{c}(\bar{x},\bar{\lambda}) \ge
\langle D_x L_{c}(\bar{x}, \bar{\lambda} ),  x -\bar{x} \rangle_{X^\ast,X}=0\quad \forall x \in X.
\end{equation*}
Similarly, by the concavity of $L_{c}(x,\cdot)$, we deduce $L_{c}(\bar{x},{\lambda})\le
L_{c}(\bar{x},\bar{\lambda})$.

Conversely, suppose that $(\bar{x},\bar{\lambda})$ is a saddle point. 
The second inequality indicates that $\bar{x}$ is a minimizer of the function $L_{c}(\cdot,\bar{\lambda})$, which implies that $D_xL_{c}(\bar{x},\bar{\lambda}) = 0$.
The similar argument shows that $D_{\lambda}L_{c}(\bar{x}, \bar{\lambda}) = 0$.
\end{proof}

\begin{corollary}
If one of the conditions in Theorem~\ref{thm:sol_multi} holds, then $\bar{x}$ is a solution of problem~\eqref{min}.
\end{corollary}
\begin{proof}
Assume that there exists a pair $(\bar{x},\bar{\lambda})$ satisfying the optimality system~\eqref{eqn:extremal}.
The system implies that $0\in D_xf(\bar{x}) + E^{\rm T}\partial\phi(E\bar{x})$.
By \cite[Chap.~1, Prop.~5.7]{Ekeland+Temam-Convanalvariprob:99}) we have
\begin{equation*}
E^{\rm T}\partial\phi(E x)\subset \partial(\phi\circ E)(x),\quad \forall x\in X.
\end{equation*}
Therefore it follows that
\begin{align*}
  0\in D_xf(\bar{x}) + E^{\rm T}\partial\phi(E\bar{x})\subset D_x f(\bar{x}) +
\partial(\phi\circ E)(\bar{x}) = \partial(f+\phi\circ E)(\bar{x}),
\end{align*}
which shows that $\bar{x}$ is a solution of the minimization problem~\eqref{min}.
\end{proof}
\begin{remark}
We refer to \cite[Chap.~4]{Ito+Kunisch-Lagrmultapprvari:08} for a sufficient condition for the existence of a pair satisfying the optimality system~\eqref{eqn:extremal}.\end{remark}
\begin{corollary}
The Lagrange optimality system can also be written as
\begin{equation}\label{KKT3}
  D_xf(\bar{x}) + E^{\rm T} \bar{\lambda}=0 \quad\mbox{and}\quad
  E \bar{x} -\prox{\phic}(E \bar{x}+\bar{\lambda}/c)=0.
\end{equation}
\end{corollary}
\begin{proof}
It follows directly from Proposition~\ref{prop:KKT},~\eqref{KKT2} and~\eqref{Lx}.
\end{proof}
The Lagrange optimality system~\eqref{KKT} is closely related to the optimality system derived in
\cite{Ito+Kunisch-AugmLagrmethnons:00,Ito+Kunisch-actistrabaseaugm:99} which is given by using \textit{the generalized Moreau-Yosida
approximation} $\psi_c(z,\lambda)$ defined by
 \begin{equation*}
 \psi_c(z,\lambda) = \phi_c(z+\lambda/c)-\tfrac{1}{2c}\|\lambda \|^2.
\end{equation*}
Let us assume that a pair $(\bar{x},\bar{\lambda})\in X\times Z$ satisfies the optimality system~\eqref{eqn:extremal}.
It is shown in \cite[Thm.~4.5]{Ito+Kunisch-AugmLagrmethnons:00} that the pair satisfies the following
optimality condition for every $c>0$.
\begin{equation*}
 \bar{x} =  \min_{x}L_{c}(x,\bar{\lambda}) \quad \mbox{and}\quad
  \bar{\lambda}  = (D_{x}\psi_c)(E \bar{x},\bar{\lambda}). 
\end{equation*}
The first relation implies the inequality $L_c(\bar{x},\bar{\lambda})
\le L_c(x,\bar{\lambda})$ for all $x\in X$, which is the second inequality of~\eqref{eqn:saddle}. Meanwhile,
by the definition of $\psi_c(x,\lambda)$ and Proposition~\ref{prop:basic}{\rm (d)}, we have
\begin{equation*}
  \begin{aligned}
  (D_x\psi_c)(E x,\lambda)& = \phi^{\prime}_c(E x + \lambda/c)\\
   &= c( E x +\lambda/c -\prox{\phic}(E x +\lambda/c ))\\
   &=\lambda + c( E x -\prox{\phic}(E x + \lambda/c)).
 \end{aligned}
\end{equation*}
In view  of the expression~\eqref{KKT2}, the second relation implies $D_{\lambda}L_{c}(\bar{x},\bar{\lambda})=0$,
which is the second equation of the Lagrange optimality system~\eqref{KKT}.
Alternatively, the following optimality condition in the form of equation is given in \cite{Ito+Kunisch-actistrabaseaugm:99}:
\begin{equation*}
\begin{aligned}
  D_xf(\bar{x}) + E^{\rm T} \bar{\lambda} = 0 \quad \mbox{and}\quad
  \bar{\lambda} =  (D_x\psi_c)(E \bar{x},\bar{\lambda}).
\end{aligned}
\end{equation*}
Similarly, one can show that this optimality system is equivalent to~\eqref{KKT3}.
\section{Linear Newton method for the Lagrange optimality system}\label{sec:newton}
In this section, we present a linear Newton method for the nonsmooth optimization problem~\eqref{min} on the basis of the Lagrange optimality system.
We also illustrate the method
on {two} elementary examples. 
To keep the presentation simple, we restrict our discussions to finite-dimensional spaces.
\subsection{Linear Newton method}
We begin with the concept of linear Newton approximation, which provides a building block for designing
Newton type algorithms for problem~\eqref{min}.
For a comprehensive treatment and for further references on the subject one may refer to \cite{Facchinei+Pang-Finivariineqcomp:03a}.
\begin{definition}\label{def:linear Newton}
Let $\Phi\colon \R^m\rightarrow \R^n$ be locally Lipschitz continuous. We say that the map $\Phi$ admits a \textit{linear Newton
approximation (LNA)} at $\bar{\xi}\in\R^m$ if there exists a set-valued map $T\colon \R^m\rightrightarrows \R^{n\times m}$
such that:
\begin{enumerate}
  \item[{\rm (a)}] The set of matrices $T(\xi)$ is nonempty and compact for each $\xi\in\R^m$;
  \item[{\rm (b)}] $T$ is upper semicontinuous at $\bar{\xi}$;
  \item[{\rm (c)}] The following limit holds:
  \begin{equation*}
    \lim_{\substack{\bar{\xi}\neq \xi \rightarrow \bar{\xi}\\ V\in T(\xi) }}
    \frac{ \Vert \Phi(\xi) + V(\bar{\xi}-\xi) - \Phi(\bar{\xi})\Vert}{\Vert \xi -\bar{\xi}\Vert} = 0.
\end{equation*}
\end{enumerate}
We also say that $T$ is a \textit{linear Newton approximation scheme} of $\Phi$.
\end{definition}
A linear Newton iteration for solving the nonlinear equation $\Phi(\xi)=0$ is defined by
\begin{equation}\label{gn}
     \xi^{k+1}= \xi^{k} - V^{-1}_k\Phi(\xi^k) , \mbox{ with }V_k \in T(\xi^k).
\end{equation}
The local convergence of the iterate is ensured if the matrix $V_k$ is nonsingular for all $k$.
\begin{theorem}[\protect{\cite[Thm.~7.5.15]{Facchinei+Pang-Finivariineqcomp:03a}
}]\label{thm:converge}
Let $\Phi \colon \R^n\rightarrow \R^n$ be locally Lipschitz continuous and admit a LNA
$T$ at $\xi^\ast\in \R^n$ such that $\Phi(\xi^\ast)=0$. If every matrix
$V\in T(\xi^\ast)$ is nonsingular, then the iterate~\eqref{gn} converges superlinearly
to the solution $\xi^\ast$ provided that $\xi^0$ is sufficiently close to $\xi^\ast$.
\end{theorem}
In addition to the Newton iteration~\eqref{gn} we can also define inexact version of linear Newton methods, the Levenberg-Marquardt (LM) method and the inexact version of LM method, and
establish their local convergence as well as characterize their convergence rate, see. e.g., \cite{Facchinei+Pang-Finivariineqcomp:03a}.
The linear Newton method for the Lagrange optimality system, which we shall develop later in the section, can be extended for these methods along similar lines,
but we restrict ourselves to the basic Newton method~\eqref{gn}.

To provide a class of Lipschitz maps that admit a LNA, we shall make use of the notion of generalized Jacobian and semismoothness.
Let $\Phi \colon \R^m\rightarrow \R^n$ be a locally Lipschitz continuous map. Rademacher's
Theorem \cite[Sect.~3.1.2]{Evans+Gariepy-Meastheofineprop:92} states that a locally continuous map is
differentiable almost everywhere. Denote by $N_{\Phi}$ a set of measure zero such that $\Phi$ is differentiable
on $\R^m\setminus N_{\Phi}$. The \textit{limiting Jacobian} of $\Phi$ at $\xi$ is the set
\begin{equation*}
\partial_B \Phi(\xi) := \left\{G \in \R^{n\times m} \mid  \exists \{\xi^k\} \subset \R^m\setminus N_{\Phi} \mbox{ with } \xi^k\rightarrow \xi,
D_x\Phi(\xi^k) \rightarrow G\right\}.
\end{equation*}
The \textit{(Clarke's) generalized Jacobian} $\partial\Phi(\xi)$ of $\Phi$ at $\xi \in\R^m$ is the convex hull of the
limiting Jacobian:
\begin{equation*}
\partial \Phi(\xi) = \mbox{conv}(\partial_B\Phi(\xi)).
\end{equation*}
We denote by $\partial_B \Phi$ the set valued map $\xi \rightarrow \partial_B\Phi(\xi)$ for $\xi\in\R^m$. The set valued map $\partial\Phi$ for the generalized Jacobian is defined analogously.

A possible choice for a LNA scheme of a locally Lipschitz map is the limiting or generalized Jacobian of the map.
This attempt, in the absence of additional assumption on $\Phi$, is doomed
because both of them do not necessarily satisfy the approximation property of condition {\rm (c)} in Definition~\ref{def:linear Newton}.
This drawback can be ameliorated by employing the notion of semismoothness, which narrows down the class of Lipschitz maps so that
each of $\partial \Phi$ and $\partial_B\Phi$ provides a LNA scheme of the map.  
\begin{definition}
Let $\Phi \colon \R^m\rightarrow \R^n$ be a locally Lipschitz map. We say that $\Phi$ is \textit{semismooth} at $\bar{\xi} \in\R^m$ if $\Phi$ is directionally differentiable near $\bar{\xi}$ and the following limit holds:
\begin{equation*}
\lim_{
\bar{\xi}\neq \xi \rightarrow \bar{\xi} \rightarrow 0
}
\frac{ \Vert \Phi^\prime(\xi;\xi-\bar{\xi}) - \Phi^\prime(\bar{\xi};\xi-\bar{\xi}) \Vert}{\Vert \xi -\bar{\xi}\Vert} = 0,
\end{equation*}
where $\Phi^\prime(\xi;h)$ denotes the directional derivative of $\Phi$ at $\xi\in \R^m$ along the direction $h\in \R^m$.
\end{definition}

\begin{proposition}\label{B_LNA}
Assume that a locally Lipschitz map $\Phi \colon \R^m\rightarrow \R^n$ is semismooth at $\xi \in \R^m$, then
each of $\partial \Phi$ and $\partial_B \Phi$ defines a LNA scheme of $\Phi$ at $\xi$.
\end{proposition}
\begin{proof}
It follows from \cite[Prop.~7.1.4]{Facchinei+Pang-Finivariineqcomp:03a} that the set valued map $\partial \Phi $ satisfies the condition {\rm (a)} and {\rm (b)} of Definition~\ref{def:linear Newton}, while, from \cite[Thm.~7.4.3]{Facchinei+Pang-Finivariineqcomp:03a}, the map satisfies the condition {\rm (c)}. We refer the proof for the limiting Jacobian to \cite[Prop.~7.5.16]{Facchinei+Pang-Finivariineqcomp:03a}.
\end{proof}
\subsection{Linear Newton method for the Lagrange optimality system}
We are ready to present a Newton algorithm for the Lagrange optimality system. Let the map $\Phi_c\colon \R^n\times \R^m \rightarrow \R^{n+m}$ be defined by
\begin{equation*}
  \Phi_c(x,\lambda)=   \begin{bmatrix}
   D_x L_{c}(x,\lambda)\\    D_{\lambda}L_{c}(x,\lambda)
  \end{bmatrix}.
\end{equation*}
Proposition~\ref{prop:KKT} shows that the map $\Phi_c$ is the difference of a smooth and nonsmooth part
\begin{equation*}
  \Phi_c(x,\lambda)  = \Phi_s(x,\lambda)   - \Phi_{ns}(x,\lambda),
\end{equation*}
where 
\begin{equation*}
  \Phi_s(x,\lambda):=
  \begin{bmatrix}
   D_x f(x) + cE^{\rm T} E x  +E^{\rm T} \lambda\\
   E x
  \end{bmatrix}
  \quad \mbox{and} \quad
  \Phi_{ns}(x,\lambda)=
  \begin{bmatrix}
   cE^{\rm T}\prox{\phic}(E x +\lambda/c) \\
  \prox{\phic}(E x +\lambda/c)
  \end{bmatrix}   .
\end{equation*}
The Jacobian of $\Phi_{s}(x,\lambda)$ is
\begin{equation*}
D_{x,\lambda}\Phi_{s}(x,\lambda) =
  \left[
    \begin{array}{c;{2pt/2pt}c}
      D^2_x f(x) + cE^{\rm T} E  & E^{\rm T}  \\ \hdashline[2pt/2pt]
   E & 0
    \end{array}
\right],
\end{equation*}
and the (matrix valued) map $D_{x,\lambda}\Phi_s$ defines a LNA scheme of the smooth map $\Phi_s$ at every point $(x,\lambda)$.
By the sum rule (see, e.g., \cite[Thm.~7.5.18]{Facchinei+Pang-Finivariineqcomp:03a}), a LNA scheme of
$\Phi_c$ is provide by $T= D_{x,\lambda}\Phi_{s} - T_{ns}$ where
$T_{ns}$ is a LNA scheme of $\Phi_{ns}$. 
The next result shows that the task of determining $T_{ns}$ is reduced to the one of computing a LNA scheme of the proximity operator.
\begin{lemma}\label{lem:linearNewton_prox0}
Let $\phi\in \Gamma_0(\R^m)$ and $c>0$. Let $T_{p}$ be a LNA scheme of the proximity operator $\prox{\phic}$. Then the set-valued map
\begin{equation*}
 T_{ns}(x,\lambda):= \left\{
 \begin{bmatrix}
   cE^{\rm T} \\ I
  \end{bmatrix}
 G\begin{bmatrix}
 E  &  c^{-1}I
 \end{bmatrix}
 \mid
 G \in T_p(E x + \lambda/c) \right\} \subset \R^{n+m,n+m}
\end{equation*}
is a LNA of the map $\Phi_{ns}$.
\end{lemma}
\begin{proof}
Since $T_p$ is upper semi-continuous and the set $T_p(z)$ is compact by definition, so is the set-valued map $(x,\lambda)\rightarrow T_{ns}(x,\lambda)$,
which implies that the $T_{ns}$ satisfies the conditions {\rm (a)} and {\rm (b)} in Definition.~\ref{def:linear Newton}. One can verify that the set valued map $T_{ns}$ satisfies the condition {\rm (c)} in the definition by employing the sum rule (\cite[Thm.~7.5.18]{Facchinei+Pang-Finivariineqcomp:03a}) and the chain rule (\cite[Thm.~7.5.17]{Facchinei+Pang-Finivariineqcomp:03a}).
\end{proof}

We now turn our attention to define a possible LNA scheme of
a proximity operator.
By Proposition~\ref{prop:basic}, the proximity operator is nonexpansive, and therefore it is
Lipschitz continuous. Hence the limiting Jacobian $\partial_B(\prox{\phi/c})(z)$ is well-defined for all $z\in \R^m$, and so also is the generalized Jacobian $\partial(\prox{\phi/c})(z)$. The next result, due to \cite[Thm.~3.2]{Patrinos+StellaETAL-ForwtrunNewtmeth:14}, gives the basic properties of the generalized Jacobian of the proximity operator.
\begin{proposition}\label{prop:jacobian}
For any $\phi \in \Gamma_0(\R^m)$, every $G \in \partial(\prox{\phic})(z)$ is a symmetric positive semidefinite matrix with
$\Vert G\Vert \le 1$.
\end{proposition}
Now we can specify a LNA scheme of the map $D_{x,\lambda}L_c$ at $(x,\lambda)$.
\begin{proposition}\label{prop:GJ}
Let $\phi\in \Gamma_0(\R^m)$ and $c>0$. Assume that the proximity operator $\prox{\phic}$ is semismooth.
Then the set-valued map $T\colon \R^n\times \R^m \rightrightarrows \R^{(n + m)\times(n+m)}$ defined by
\begin{equation}\label{partial_Phi_c}
T(x,\lambda) :=\left\{
  \left[
    \begin{array}{c;{2pt/2pt}c}
      D^2_x f(x) + cE^{\rm T} (I- G)E  & ((I- G)E)^{\rm T}  \\ \hdashline[2pt/2pt]
    (I- G)E & -c^{-1} G
    \end{array}
\right]
 \mid G\in \partial(\prox{\phic})(z)\right\},
\end{equation}
with $z=E x +\lambda/c$, is a LNA scheme of the map $\Phi_c$ at $(x,\lambda)\in \R^n\times\R^m$.
\end{proposition}
\begin{proof}
The symmetry of the generalized Jacobian of a proximity operator allows to write
$E^{\rm T} G = (GE)^{\rm T}$ for $G\in\partial (\prox{\phic})(E x + \lambda/c)$, which yields
\begin{align*}
  \left[
    \begin{array}{c;{2pt/2pt}c}
      D^2_x f(x) + cE^{\rm T} (I- G)E  & ((I- G)E)^{\rm T}  \\ \hdashline[2pt/2pt]
    (I- G)E & -c^{-1} G
    \end{array}
\right]=
  \left[
    \begin{array}{c;{2pt/2pt}c}
      D^2_x f(x) + cE^{\rm T} E  & E^{\rm T}  \\ \hdashline[2pt/2pt]
   E & 0
    \end{array}
\right]
-
\begin{bmatrix}
   cE^{\rm T} \\ I
\end{bmatrix}
G
\begin{bmatrix}
 E  &  c^{-1}I
\end{bmatrix}.
\end{align*}
From Proposition~\ref{B_LNA} and the assumption that $\prox{\phic}$ is semismooth, it follows that the generalized Jacobian
$\partial(\prox{\phic})(z)$ is a LNA scheme of the proximity operator $\prox{\phic}(z)$, which together with Lemma~\ref{lem:linearNewton_prox0} shows that $T_{ns}(x,\lambda)$ with $T_p(E x +\lambda/c)=\partial(\prox{\phic})(E x +\lambda/c)$ defines a LNA scheme of $\Phi_{ns}$ at $(x,\lambda)$. Thus $T=D_{x,\lambda}\Phi_s - T_{ns}$ defines a LNA scheme of $\Phi_c$ at $(x,\lambda)$.
\end{proof}
\begin{remark}
One can replace the generalized Jacobian $\partial(\prox{\phic})(z)$ in~\eqref{partial_Phi_c} with the limiting Jacobian $\partial_B(\prox{\phic})(z)$.
\end{remark}
\begin{remark}
The class of semismooth maps is broad enough to include a variety of proximity operators frequently encountered in practice, see, e.g., \cite[Sect.~5]{Patrinos+StellaETAL-ForwtrunNewtmeth:14}.
\end{remark}

The proposed algorithm is given in Algorithm~\ref{alg:ssn1}.
\begin{algorithm}[H]
  \caption{Linear Newton algorithm for the Lagrange optimality system.}\label{alg:ssn1}
  \begin{algorithmic}[1]
\STATE Chose $(x^0,\lambda^0)\in\R^n\times\R^m$.
 \STATE If $\Phi_c(x^k,\lambda^k)=0$, stop.
 \STATE Let $z^k=E x^k +\lambda^k/c$, and compute an element $G_k$ of the generalized Jacobian $\partial({\prox{\phic}})(z^k)$.
 \STATE Compute a direction $(d^k_x,d^k_{\lambda})$ by
\begin{equation}\label{Newton1}
  \left[
    \begin{array}{c;{2pt/2pt}c}
      D^2_x f(x^k) + cE^{\rm T} (I- G_k)E  & ((I- G_k)E)^{\rm T}  \\ \hdashline[2pt/2pt]
    (I- G_k)E & -c^{-1} G_k
    \end{array}
\right]
   \begin{bmatrix}
   d^k_x \\   d^k_\lambda
  \end{bmatrix}   =
   -\begin{bmatrix}
    D_xL_c(x^k,\lambda^k)\\
    D_{\lambda}L_c(x^k,\lambda^k)
  \end{bmatrix}.
  \end{equation}
\STATE Set $x^{k+1} = x^k + d^k_x$ and $\lambda^{k+1}=\lambda^k + d^k_{\lambda}$.
\STATE Go back to Step 2.
\end{algorithmic}
\end{algorithm}
\begin{remark}
Proposition~\ref{B_LNA} allows to replace the generalized Jacobian $\partial(\prox{\phic})(z)$ with the limiting Jacobian $\partial_B(\prox{\phic})(z)$.
\end{remark}
\begin{remark}
A simple calculation using Theorem~\ref{prop:KKT} shows that the update at Steps 4 and 5 can be replaced with
\begin{align}\label{Newton2}
\left[
    \begin{array}{c;{2pt/2pt}c}
        D^2_xf(x^k) & E^{\rm T}  \\ \hdashline[2pt/2pt]
 (I- G_k)E & -{c^{-1}}G_k
    \end{array}
\right]
\begin{bmatrix}
x^{k+1}\\
\lambda^{k+1}
\end{bmatrix}
=
\begin{bmatrix}
  D_x^2f(x^k)x^k -D_xf(x^k)\\
  \prox{\phic}(z^k)-G_kz^k
  \end{bmatrix}.
\end{align}
\end{remark}
The local convergence of Algorithm~\ref{alg:ssn1} follows from Theorem~\ref{thm:converge}, if
every element of $T(x,\lambda)$ defined by~\eqref{partial_Phi_c} is nonsingular.
The next result gives one sufficient condition for the nonsingularity.
\begin{proposition}\label{prop:nonsingular}
Assume that $E$ is surjective, $D^2f(x)$ is strictly positive definite, and
the norm is bound from below uniformly in $x$, that is, there exists a $\delta>0$ such that
\begin{equation*}
(D^2_xf(x)d,d)> \delta \Vert d\Vert^2\quad \forall d\in \R^n.
\end{equation*}
Then every element of $T(x,\lambda)$ is nonsingular for all $(x,\lambda)$.
\end{proposition}
\begin{proof}
A saddle point matrix of the form
\begin{equation*}
\begin{bmatrix}
A & B^{\rm T} \\
B & -C
\end{bmatrix},
\end{equation*}
where $A$ is symmetric positive definite and $C$ is symmetric positive semidefinite,
is nonsingular {if} $\mbox{ker}(C)\cap \mbox{ker}(B^{\rm T})=0$, see, e.g., \cite[Thm.~3.1]{Benzi+GolubETAL-Numesolusaddpoin:05}. Note that $D^2_xf(x)$ is symmetric positive
definite by assumption, and $G$ and $I-G$ are symmetric positive semidefinite, cf.
Proposition~\ref{prop:jacobian}. Hence the matrix $ D^2_xf(x) + cE^{\rm T}(I-G)E$
is symmetric positive definite. Now let $d \in \mbox{ker}(G)\cap \mbox{ker}(((I-G)
E)^{\rm T})$. We then have
\begin{equation*}
Gd=0\quad \mbox{and}\quad  ((I-G)E)^{\rm T} d =0.
\end{equation*}
Appealing again to the identity $G^*=G$ from Proposition~\ref{prop:jacobian}, it immediately follows that
$
E^{\rm T} d = 0.
$
Then the surjectivity of $E$ implies $d\in \mbox{ker}(E^{\rm T}) = \mbox{Im}(E)^{\bot} = 0$.
\end{proof}
The local convergence of Algorithm 1 follows from Theorem~\ref{thm:converge},
Propositions~\ref{prop:GJ} and Proposition~\ref{prop:nonsingular}.
\begin{theorem}\label{thm:local}
Let $f$ be smooth, $\phi\in \Gamma_0(\R^m)$, and $c>0$. Let us assume there exits a unique solution $(\bar{x},\bar{\lambda})$
 of the Lagrange optimality system \eqref{KKT}. We also assume that the assumptions on $f$ and $E$ in Proposition~\ref{prop:nonsingular} are satisfied,
and that the proximity operator is semismooth on $\R^m$. Then the Newton system~\eqref{Newton1} is
solvable, and the sequence $(x^k,\lambda^k)$ generated by Algorithm~\ref{alg:ssn1} converges to the
solution $(\bar{x},\bar{\lambda})$ superlinearly in a neighborhood of $(\bar{x},\bar{\lambda})$.
\end{theorem}

\subsection{Examples}
We illustrate Algorithm~\ref{alg:ssn1} on {two} examples: bilateral constraints and 
$\ell^1$ penalty. 
We begin with a useful result for computing the generalized (limiting) Jacobian for (block) separable functions \cite[Prop.~3.3]{Patrinos+StellaETAL-ForwtrunNewtmeth:14}.
Let $(m_1,\ldots,m_N)$ be an $N$ partition of $m$, i.e., $\sum_{i=1}^N m_i = m$, and $z\in \R^m$ be decomposed
into $N$ blocks of variables with $z_i\in \R^{m_i}$.  The function $\phi\in \Gamma_0(\R^m)$ is said to
be \textit{(block) separable} if $\phi(z) =\sum_{i=1}^{N}\phi_i(z_i)$ for $N$ functions $\phi_i\in \Gamma_0(\R^m)$.
\begin{proposition}\label{prop:block}
  If $\phi\in \Gamma_0(\R^m)$ is (block) separable then every element of the generalized Jacobian
  $\partial(\prox{\phic})(x)$ is also a (block) diagonal matrix.
\end{proposition}
\begin{example}\label{exam:ineqconstraint}
Let us consider the following optimization problem with bilateral inequality constraints
\begin{equation*}
\min_{x\in \R^n} f(x) \quad \mbox{subject to} \quad a\le E x \le b,
\end{equation*}
where $f$ is a smooth function, $a,b \in \R^m$ and $E \in \mathbb{R}^{m\times n}$.
\end{example}

The problem can be reformulated into~\eqref{min} with $\phi(z) = I_S(z)$, where $I_S(z)$ is the
characteristic function of the set $S=\{z\in\R^m\mid a_j \le z_j \le b_j, \; j=1,\ldots,m\}$.
Clearly, the proximity operator $\prox{\phic}:\mathbb{R}^m\to\mathbb{R}^m$ is given by
\begin{equation*}
\prox{\phic}(z)
=\left[\max(a_1,\min(b_1,z_1)),\ldots,\max(a_m,\min(b_m,z_m))\right]^{\rm T}.
\end{equation*}
Since the proximity operator is separable, a limiting Jacobian
$G\in\partial_B\;{\prox{\phic}}(z)$ is diagonal matrix by Proposition~\ref{prop:block}:
\begin{equation*}
G_{j,j} = \left\{\begin{array}{cl}
    1 &  \mbox{if } a_j < z_j < b_j,\\
    \{0,1\} & \mbox{if } z_j \in \{a_j, b_j\},\\
    0 & \mbox{otherwise}.
  \end{array}\right.
\end{equation*}
Now let $(x,\lambda)$ be the current iterate, and $z=E x +\lambda/c$. We denote by  $\mo$
the index set $\{j\mid G_{j,j} =0\}\subset \{1,2,\ldots,m\}$, and by $\mi$ its complement.
Then $\mi\cap\mo = \emptyset$ and $\mi\cup\mo=\{1,2,\ldots,m\}$.
We shall denote by $x_{\mo}$ the subvector of $x$,
consisting of entries of $x$ whose indices are listed in $\mo$. The submatrix of $E$ denoted by $E_\mo$ is defined analogously. For example, if
$\mo = \{o_1,o_2,\ldots,o_{\ell}\}$ where $\ell$ is the number of elements of the set $\mo$,
then $x_{\mo} $ is $\ell\times 1$ column vector, and $A_{\mo} $ is $\ell \times n$ matrix
given respectively by
\begin{equation*}
x_{\mo} =
\begin{bmatrix}
x_{o_1} \\
x_{o_2}\\
\vdots\\
x_{o_m}
\end{bmatrix}
\quad
\mbox{ and } \quad A_{\mo} =
\begin{bmatrix}
A_{o_1,1} & A_{o_1,2} & \cdots   &A_{o_1,n} \\
A_{o_2,1} & A_{o_2,2} & \cdots   &A_{o_2,n} \\
\vdots  &            \vdots             &             & \vdots  \\
A_{o_\ell,1} & A_{o_\ell,2} & \cdots &A_{o_\ell,n}
\end{bmatrix}.
\end{equation*}
With the new updates denoted by $x^+$ and $\lambda^+$, the Newton update~\eqref{Newton2} yields
\begin{align*}
\left\{
\begin{array}{l}
\lambda^+_\mi =c(z-\prox{\phic}(z))_\mi,\\
\begin{bmatrix}
D^2_xf(x) & E_\mo^{\rm T} \\
E_\mo & 0
\end{bmatrix}
\begin{bmatrix} x^+ \\ \lambda^+_\mo
\end{bmatrix} =
\begin{bmatrix}
\displaystyle D^2_xf(x)x-D_xf(x) - E_\mi^{\rm T} \lambda^+_\mi \\
\prox{\phic}(z)_\mo
\end{bmatrix}.
\end{array}\right.
\end{align*}
In this example, we have $z_{\mi}=\prox{\phic}(z)_{\mi}$, and the Newton update is further simplified as
\begin{align*}
&\begin{bmatrix}
D^2_xf(x) & E_\mo^{\rm T} \\
E_\mo &0
\end{bmatrix}
\begin{bmatrix}
x^+ \\ \lambda^+_\mo
\end{bmatrix}
= \begin{bmatrix}
\displaystyle D^2_xf(x)x-D_xf(x)\\
\prox{\phic}(z)_\mo
\end{bmatrix}\qquad \mbox{and}\qquad
\lambda^+_\mi=0.
\end{align*}
In particular if $f$ is a quadratic function $f(x) = \frac{1}{2}(x,Ax)-(b,x)$, the algorithm reduces to the
primal-dual active set algorithm developed in \cite{Ito+Kunisch-Lagrmultapprvari:08,Bergounioux+ItoETAL-Primstraconsopti:99}:
\begin{align*}
&\begin{bmatrix}
A & E_\mo^{\rm T} \\
E_\mo &0
\end{bmatrix}
\begin{bmatrix}
x^+ \\ \lambda^+_\mo
\end{bmatrix}
= \begin{bmatrix}  b\\
\prox{\phic}(z)_\mo
\end{bmatrix}\qquad \mbox{and}\qquad
\lambda^+_\mi=0.
\end{align*}

\begin{example}\label{exam:l1}
Consider the following $\ell^1$ type optimization problem
\begin{equation*}
\min_{x\in \R^b} f(x) + \alpha \vert E x\vert_{\ell^1},
\end{equation*}
where $f$ is smooth function, $E \in \mathbb{R}^{m\times n}$, $\vert z \vert_{\ell^1} $ is the $\ell^1$ norm,
and $\alpha>0$ is a regularization parameter.
\end{example}
Let $\phi(z) = \alpha \vert z \vert_{\ell^1}$. Its proximity operator $\prox{\phic}$ is the well known soft-thresholding operator
\begin{equation*}
  \begin{aligned}
     \prox{\phic}(z) &= [\prox{\frac{ \alpha}{c}\vert \cdot \vert}(z_1), \ldots,\prox{\frac{ \alpha}{c}\vert \cdot \vert}(z_m)]^{\rm T},\\
     \prox{\frac{ \alpha}{c}\vert \cdot \vert}(s) &= \max(s-\tfrac{ \alpha}{c} ,\min(s+\tfrac{ \alpha}{c} ,0)),\quad s\in \R.
  \end{aligned}
\end{equation*}
A limiting Jacobian $G\in \partial_B(\prox{\phic})(z)$ is diagonal matrix given by
\begin{equation*}
G_{j,j} =\left\{
\begin{array}{ll}
1 &  \mbox{ if } \vert z_j\vert  > \frac{ \alpha}{c} ,\\
\{0,1\} & \mbox{ if } \vert z_j \vert= \frac{ \alpha}{c}, \\
0 & \mbox{ othewise. }
\end{array} \right.
\end{equation*}
We denote by  $\mo$ the index set $\{j\mid G_{j,j} =0\}\subset \{1,2,\ldots,m\}$, and by $\mi$
its complement, and $z=E x +\lambda/c$. We note that the relation $c(z-\prox{\phic}(z))_\mi=
c\;{\rm sign}(z_\mi)$ holds. An argument similar to Example~\ref{exam:ineqconstraint} yields the following Newton update
\begin{align*}
\left\{
\begin{array}{l}
\lambda^+_\mi= c\;{\rm sign}(z_\mi), \\[5pt]
\begin{bmatrix}
D^2_xf(x) & E_\mo^{\rm T} \\
E_\mo &0
\end{bmatrix}
\begin{bmatrix}
x^+ \\ \lambda^+_\mo
\end{bmatrix}
= \begin{bmatrix}
\displaystyle D^2_xf(x)x-D_xf(x) -E_\mi^{\rm T} \lambda^+_\mi \\
\prox{\phic}(z)_\mo
\end{bmatrix}.
\end{array}
\right.
\end{align*}
For the quadratic function $f=\frac{1}{2}(x,Ax)-(b,x)$, we obtain a primal-dual active set algorithm for $\ell^1$ norm regularization
\begin{equation*}%\label{ell1:newton}
\left\{\begin{array}{l}
\lambda^+_\mi=c\;{\rm sign}(z_\mi),\\[5pt]
\begin{bmatrix}
A & E_\mo^{\rm T} \\
E_\mo &0
\end{bmatrix}
\begin{bmatrix}
x^+ \\
\lambda^+_\mo
\end{bmatrix}
=\begin{bmatrix}
\displaystyle b -E_\mi^{\rm T} \lambda^+_\mi\\
\prox{\phic}(z)_\mo
\end{bmatrix}.
\end{array}
\right.
\end{equation*}
\section{Conclusion}
In this paper, we have developed the classical Lagrange multiplier approach to a class of nonsmooth convex optimization problems
arising in various application domains.
We presented the Lagrange optimality system, and established the equivalence among the Lagrange optimality system, the standard optimality condition and the saddle point condition of the augmented Lagrangian. 
The Lagrange optimality system was used to derive a novel Newton algorithm. We proved the nonsingularity of the Newton system and established the local convergence of the algorithm.

In order to make the proposed Newton algorithm applicable to real word applications, a further study is needed on several important issues including: to construct a merit function for the globalization of the algorithm; to develop efficient solvers for the (possibly) large linear system (Newton update); to provide a stopping criterion, and to report the numerical performance of the algorithm. These issues will be investigated in future work.


\begin{thebibliography}{spmpsci}
\bibitem{boyd2011distributed}
Boyd, S., Parikh, N., Chu, E., Peleato, B., Eckstein, J.:
Distributed optimization and statistical learning via the alternating direction method of multipliers.
Found. Trends Mach. Learning \textbf{3}, 1--122 (2011)

\bibitem{ChanShen:2005}
Chan, T.F., Shen, J.:
Image {P}rocessing and {A}nalysis.
SIAM, Philadelphia, PA (2005)

\bibitem{Combettes+Wajs-Signrecoproxforw:05}
Combettes, P.L., Wajs, V.R.:
Signal recovery by proximal forward-backward splitting. Multiscale Model.
Simul. \textbf{4}, 1168--1200 (2005)

\bibitem{Ekeland+Temam-Convanalvariprob:99}
Ekeland, I., T{\'e}mam, R.:
Convex {A}nalysis and {V}ariational {P}roblems.
SIAM, Philadelphia, PA (1999)

\bibitem{Glowinski-Numemethnonlvari:08}
Glowinski, R.:
Numerical {M}ethods for {N}onlinear {V}ariational {P}roblems.
Springer-Verlag, Berlin (2008)

\bibitem{ItoJin:2014}
Ito, K., Jin, B.:
Inverse {P}roblems: {T}ikhonov {T}heory and {A}lgorithms.
World Scientific, Singapore (2014)

\bibitem{Ito+Kunisch-Lagrmultapprvari:08}
Ito, K., Kunisch, K.:
Lagrange {M}ultiplier {A}pproach to {V}ariational {P}roblems and {A}pplications.
SIAM, Philadelphia, PA (2008)

\bibitem{Hestenes-Multgradmeth:69}
Hestenes, M.R.:
Multiplier and gradient methods. J. Optim. Theory Appl. \textbf{4}, 303--320 (1969)

\bibitem{Powell-methnonlconsmini:69}
Powell, M.J.D.:
A method for nonlinear constraints in minimization problems.
In: Optimization ({S}ympos., {U}niv. {K}eele, {K}eele, 1968), pp. 283--298. Academic Press, London (1969)

\bibitem{Rockafellar-dualapprsolvnonl:73}
Rockafellar, R.:
A dual approach to solving nonlinear programming problems by unconstrained optimization.
Math. Programming \textbf{5}, 354--373 (1973)

\bibitem{Rockafellar-multmethHestPowe:73}
Rockafellar, R.:
The multiplier method of {Hestenes and Powell} applied to convex programming.
J. Optim. Theory Appl. \textbf{12}, 555--562 (1973)

\bibitem{glowinski1975}
Glowinski, R., Marroco, A.:
Sur l'approximation, par {\'e}l{\'e}ments finis d'ordre un, et la r{\'e}solution, par p{\'e}nalisation-dualit{\'e} d'une classe de probl{\`e}mes de dirichlet non lin{\'e}aires.
ESAIM: Math. Model. Numer. Anal. \textbf{9}, 41--76 (1975)

\bibitem{GlowinskiLeTallec:1989}
Glowinski, R., Le~Tallec, P.:
Augmented {L}agrangian and {O}perator-{S}plitting {M}ethods in {N}onlinear {M}echanics.
SIAM, Philadelphia, PA (1989)

\bibitem{parikh2013proximal}
Parikh, N., Boyd, S.:
Proximal algorithms. Found.
Trends Optim. \textbf{1}, 123--231 (2013)

\bibitem{WuTai:2010}
Wu, C., Tai, X.C.:
Augmented {L}agrangian method, dual methods, and split {B}regman iteration for {ROF}, vectorial {TV}, and high order models.
SIAM J. Imaging Sci. \textbf{3}, 300--339 (2010)

\bibitem{Fortin-Minisomenon-func:75}
Fortin, M.:
Minimization of some non-differentiable functionals by the augmented {L}agrangian method of {H}estenes and {P}owell.
Appl. Math. Optim. \textbf{2}, 236--250 (1975)

\bibitem{Ito+Kunisch-AugmLagrmethnons:00}
Ito, K., Kunisch, K.:
Augmented lagrangian methods for nonsmooth, convex optimization in {H}ilbert spaces.
Nonlin. Anal. Ser. A Theory Methods \textbf{41}, 591--616 (2000)

\bibitem{Lemarechal+Sagastizabal-PracaspeMoreregu:97}
Lemar{\'e}chal, C., Sagastiz{\'a}bal, C.:
Practical aspects of the {M}oreau-{Y}osida regularization: theoretical preliminaries.
SIAM J. Optim. \textbf{7}, 367--385 (1997)

\bibitem{Ip+Kyparisis-Locaconvquasmeth:92}
Ip, C.M., Kyparisis, J.:
Local convergence of quasi-{N}ewton methods for {$B$}-differentiable equations.
Math. Programming \textbf{56}, 71--89 (1992)

\bibitem{Facchinei+Pang-Finivariineqcomp:03a}
Facchinei, F., Pang, J.S.:
Finite-{D}imensional {V}ariational {I}nequalities and {C}omplementarity {P}roblems.
{V}ol. {II}. Springer-Verlag, New York (2003)

\bibitem{Bauschke+Combettes-Convanalmonooper:11}
Bauschke, H.H., Combettes, P.L.:
Convex {A}nalysis and {M}onotone {O}perator {T}heory in {H}ilbert {S}paces.
Springer, New York (2011)

\bibitem{Ito+Kunisch-actistrabaseaugm:99}
Ito, K., Kunisch, K.:
An active set strategy based on the augmented {L}agrangian formulation for image restoration.
ESAIM: Math. Model. Numer. Anal. \textbf{33}, 1--21 (1999)

\bibitem{Evans+Gariepy-Meastheofineprop:92}
Evans, L.C., Gariepy, R.F.:
Measure {T}heory and {F}ine {P}roperties of {F}unctions.
CRC Press, Boca Raton, FL (1992)

\bibitem{Patrinos+StellaETAL-ForwtrunNewtmeth:14}
Patrinos, P., Stella, L., Bemporad, A.:
Forward-backward truncated {N}ewton methods for convex composite optimization.
preprint, arXiv:1402.6655v2 (2014)

\bibitem{Benzi+GolubETAL-Numesolusaddpoin:05}
Benzi, M., Golub, G.H., Liesen, J.:
Numerical solution of saddle point problems.
Acta Numer. \textbf{14}, 1--137 (2005)

\bibitem{Bergounioux+ItoETAL-Primstraconsopti:99}
Bergounioux, M., Ito, K., Kunisch, K.:
Primal-dual strategy for constrained optimal control problems.
SIAM J. Control Optim. \textbf{37}, 1176--1194 (1999)

\end{thebibliography}
\end{document}